\documentclass[a4paper,12pt]{article}

\usepackage[utf8]{inputenc}
\usepackage[english]{babel}
\usepackage{amsmath, amssymb, amsthm}
\usepackage{enumerate}
\usepackage[colorlinks=true,linkcolor=blue,citecolor=blue]{hyperref}

\usepackage{tikz}
\usetikzlibrary{fadings}
\tikzfading[name=fade out,
            inner color=transparent!0,
            outer color=transparent!100]
\usetikzlibrary{calc,arrows,decorations.pathreplacing,fadings,3d,positioning}

\title{Multiplicative parametric geometry of numbers \\ and \\ transference theorems for lattice exponents.
       \thanks{This research was supported in part by RSF grant 18-41-05001}}
\author{Oleg\,N.\,German}
\date{}

\theoremstyle{definition}
\newtheorem{definition}{Definition}

\newtheorem*{notation*}{Notation}

\theoremstyle{remark}
\newtheorem{remark}{Remark}
\newtheorem*{remark*}{Remark}

\theoremstyle{plain}
\newtheorem{theorem}{Theorem}

\newtheorem{proposition}{Proposition}

\newtheorem*{statement*}{Statement}
\newtheorem*{corollary*}{Corollary}

\renewcommand{\phi}{\varphi}

\renewcommand{\vec}[1]{\mathbf{#1}}
\renewcommand{\geq}{\geqslant}
\renewcommand{\leq}{\leqslant}

\newcommand{\R}{\mathbb{R}}
\newcommand{\Z}{\mathbb{Z}}

\newcommand{\La}{\Lambda}

\newcommand{\bpsi}{\underline{\psi}}
\newcommand{\apsi}{\overline{\psi}}
\newcommand{\bPsi}{\underline{\Psi}}
\newcommand{\aPsi}{\overline{\Psi}}

\newcommand{\cB}{\mathcal{B}}

\newcommand{\cL}{\mathcal{L}}

\newcommand{\cT}{\mathcal{T}}

\begin{document}

\maketitle

\begin{abstract}
  In this paper we adapt parametric geometry of numbers developed by Wolfgang Schmidt and Leonard Summerer to a multiplicative setting, and derive a chain of inequalities for the corresponding exponents which splits the transference inequality for Diophantine exponents of lattices in the same way Khintchine's transference inequalities for simultaneous approximation can be split.
\end{abstract}


\section{Introduction}

A wide variety of problems in Diophantine approximation concerns approximating a given subspace of $\R^d$ with rational subspaces of a fixed dimension. Thus the concept of a \emph{Diophantine exponent} naturally arises. Say, given $\Theta=(1,\theta_1,\ldots,\theta_{d-1})\in\R^d$ and $k\in\Z$, $1\leq k\leq d-1$, the quantity
\[\omega_k(\Theta)=\sup\Big\{ \gamma\in\R \,\Big|\, |\Theta\wedge\vec Z|\leq|\vec Z|^{-\gamma}\text{ for infinitely many decomposable }\vec Z\in\textstyle\bigwedge^k(\Z^d) \Big\},\]
where $|\cdot|$ denotes the sup-norm, is called the \emph{$k$-th Diophantine exponent} of $\Theta$.

One can easily check that $\omega_1(\Theta)$ equals the supremum of $\gamma$ such that the inequality
\[\max_{1\leq i\leq d-1}|q\theta_i-p_i|\leq\max_{1\leq i\leq d-1}|p_i|^{-\gamma}\]
admits infinitely many solutions in $(q,p_1,\ldots,p_{d-1})\in\Z^d$. Similarly, $\omega_{d-1}(\Theta)$ equals the supremum of $\gamma$ such that the inequality
\[|q+\theta_1p_1+\ldots+\theta_{d-1}p_{d-1}|\leq\max_{1\leq i\leq d-1}|p_i|^{-\gamma}\]
admits infinitely many solutions in $(q,p_1,\ldots,p_{d-1})\in\Z^d$.

\paragraph{Splitting transference inequalities.}

In 1926 Khintchine \cite{khintchine_palermo} established the famous transference inequalities
\begin{equation} \label{eq:khintchine_transference}
  \begin{aligned}
    \omega_{d-1}(\Theta) & \geq(d-1)\omega_1(\Theta)+d-2, \\
    \omega_1^{-1}(\Theta) & \geq(d-1)\omega_{d-1}^{-1}(\Theta)+d-2.
  \end{aligned}
\end{equation}
In 2007 Laurent \cite{laurent_up_down} following Schmidt \cite{schmidt_annals_1967} split these inequalities into the chains
\begin{equation} \label{eq:laurent_splitting_khintchine}
  \begin{aligned}
    (d-k-1)\omega_{k+1}(\Theta) & \geq(d-k)\omega_k(\Theta)+1, \\
    k\omega_k^{-1}(\Theta) & \geq(k+1)\omega_{k+1}^{-1}(\Theta)+1,
  \end{aligned}
  \qquad
  k=1,\ldots,d-2.
\end{equation}

Similar splitting of Dyson's inequality \cite{dyson} corresponding to the case of approximating a subspace of dimension greater than 1, as well as splitting the inequalities for uniform analogues of $\omega_1$ and $\omega_{d-1}$, can be found in \cite{german_AA_2012}.

\paragraph{Diophantine exponents of lattices.}

Transference inequalities of various kinds connect problems which are dual in some sense. For instance, Khintchine's inequalities relates the problem of approximating a one-dimensional subspace of $\R^d$ with one-dimensional rational subspaces to the problem of approximating that same subspace with $(d-1)$-dimensional rational subspaces.

Recently in \cite{german_2017} a transference theorem for Diophantine exponents of lattices was proved. Let $\cL_d$ denote the space of unimodular lattices in $\R^d$. Let $\La\in\cL_d$. Set
\[\Pi(\vec x)=\prod_{1\leq i\leq d}|x_i|^{1/d}\]
for each $\vec x=(x_1,\ldots,x_d)\in\R^d$. The \emph{Diophantine exponent} of $\La$ is defined as
\[\omega(\La)=\sup\Big\{\gamma\in\R\ \Big|\,\Pi(\vec x)\leq|\vec x|^{-\gamma}\text{ for infinitely many }\vec x\in\La \Big\},\]
where $|\cdot|$ is again the sup-norm.
%
%
Consider the dual lattice
\[ \La^\ast=\Big\{ \vec y\in\R^d \,\Big|\, \langle\vec y,\vec x\rangle\in\Z\text{ for each }\vec x\in\La \Big\}, \]
where $\langle\,\cdot\,,\,\cdot\,\rangle$ is the inner product.

%

\begin{theorem}[see \cite{german_2017}] \label{t:lattice_transference}
  For each $\La\in\cL_d$ we have
  \begin{equation} \label{eq:lattice_transference}
    \omega(\La)\geq\frac{\omega(\La^\ast)}{(d-1)^2+d(d-2)\omega(\La^\ast)}\,.
  \end{equation}
  Here we mean that if $\omega(\La^\ast)=\infty$, then $\omega(\La)\geq\dfrac{1}{d(d-2)}$\,.
\end{theorem}

\begin{remark} \label{r:lattice_transference}
  Notice that \eqref{eq:lattice_transference} can be reformulated as
  \[1+\omega(\La)^{-1}\leq
    (d-1)^2(1+\omega(\La^\ast)^{-1}).\]
\end{remark}

One of the main incentives to commence this research was the wish to split \eqref{eq:lattice_transference} the way Khintchine's inequalities can be split. To this end we adapt the Schmidt--Summerer parametric geometry of numbers to the lattice setting and call it \emph{multiplicative parametric geometry of numbers}. It appears that this approach provides many natural ways to define intermediate exponents. We observe certain phenomena of local nature and use them to obtain the desired splitting. We would like to notice that though we used multiplicative parametric geometry of numbers as a mere tool, we deem it to be of interest in itself.

The structure of the paper is as follows. In Section \ref{sec:parametric} we prove basic statements of multiplicative parametric geometry of numbers, in Section \ref{sec:lattice_exponents_schimmerered} we interpret lattice exponents in terms of the constructed theory, and in Section \ref{sec:splitting} we split the reformulated Theorem \ref{t:lattice_transference}. Finally, in Section \ref{sec:3dim} we consider the three-dimensional case.

\section{Multiplicative parametric geometry of numbers} \label{sec:parametric}

\subsection{Successive minima} 

In the spirit of fundamental works \cite{schmidt_summerer_2009}, \cite{schmidt_summerer_2013}, \cite{roy_annals_2015} it is natural to propose the following approach. Most of the argument proposed in this Section is a translation to the current context of the argument of Schmidt and Summerer \cite{schmidt_summerer_2009}, and also of the paper \cite{german_AA_2012}. We remind that $|\cdot|$ denotes the sup-norm.

Let $\La\in\cL_d$. Set
\[\cB=\Big\{ \vec x\in\R^d\, \Big|\,|\vec x|\leq1 \Big\}.\]
For each $\pmb\tau=(\tau_1,\ldots,\tau_d)\in\R^d$, $\tau_1+\ldots+\tau_d=0$, set
\[D_{\pmb\tau}=\textup{diag}(e^{\tau_1},\ldots,e^{\tau_d})\]
and
\[\cB_{\pmb\tau}=D_{\pmb\tau}\cB.\]

Let $\lambda_k(\cB_{\pmb\tau})=\lambda_k(\cB_{\pmb\tau},\La)$, $k=1,\ldots,d$, denote the $k$-th successive minimum, i.e.
the infimum of positive $\lambda$ such that $\lambda\cB_{\pmb\tau}$ contains at least $k$ linearly independent vectors of $\La$.
Finally, for each $k=1,\ldots,d$, let us set
\[L_k(\pmb\tau)=L_k(\pmb\tau,\La)=\log\big(\lambda_k(\cB_{\pmb\tau},\La)\big),\qquad
  S_k(\pmb\tau)=S_k(\pmb\tau,\La)=\sum_{1\leq j\leq k}L_j(\pmb\tau,\La).\]


\subsection{Properties of $L_k$ and $S_k$} 

Since we want all the $\cB_{\pmb\tau}$ to be of the same volume, we assume the sum of the components of $\pmb\tau$ to be zero. Let us set
\[\cT=\Big\{ \pmb\tau=(\tau_1,\ldots,\tau_d)\in\R^d\, \Big|\,\tau_1+\ldots+\tau_d=0 \Big\}.\]
For each $\pmb\tau\in\cT$ set
\[|\pmb\tau|_+=\max_{1\leq k\leq d}\tau_k,
  \qquad
  |\pmb\tau|_-=|-\pmb\tau|_+=-\min_{1\leq k\leq d}\tau_k.\]
Clearly,
\[|\pmb\tau|=\max(|\pmb\tau|_-,|\pmb\tau|_+),\]
\begin{equation}\label{eq:tau_minus_vs_tau_plus}
  |\pmb\tau|_+/(d-1)\leq|\pmb\tau|_-\leq(d-1)|\pmb\tau|_+\ .
\end{equation}

\begin{proposition} \label{prop:properties_of_L_k}
  The functions $L_k(\pmb\tau)$ enjoy the following properties:

  \vskip1.5mm
  \textup{(i)} $L_1(\pmb\tau)\leq\ldots\leq L_d(\pmb\tau)$;

  \vskip1.5mm
  \textup{(ii)} $0\leq-L_1(\pmb\tau)\leq|\pmb\tau|_++O(1)$;

  \vskip1.5mm
  \textup{(iii)} $L_d(\pmb\tau)\leq|\pmb\tau|_-+O(1)$;

  \vskip1.5mm
  \textup{(iv)} each $L_k(\pmb\tau)$ is continuous and piecewise linear.
\end{proposition}

\begin{proof}
Statement \textup{(i)} follows immediately from the definition of successive minima. The inequality $L_1(\pmb\tau)\leq0$ is a corollary of Minkowski's convex body theorem. The rest of \textup{(ii)} and \textup{(iii)} is provided by
\[e^{-|\pmb\tau|_+}\cB_{\pmb\tau}\subset\cB\implies
  \lambda_1(e^{-|\pmb\tau|_+}\cB_{\pmb\tau})\geq\lambda_1(\cB)\implies
  \lambda_1(\cB_{\pmb\tau})\geq e^{-|\pmb\tau|_+}\lambda_1(\cB)\]
and
\[e^{|\pmb\tau|_-}\cB_{\pmb\tau}\supset\cB\implies
  \lambda_d(e^{|\pmb\tau|_-}\cB_{\pmb\tau})\leq\lambda_d(\cB)\implies
  \lambda_d(\cB_{\pmb\tau})\leq e^{|\pmb\tau|_-}\lambda_d(\cB).\]

Let us prove \textup{(iv)}. For each nonzero $\vec v\in\La$ let us denote by $\lambda_{\vec v}(\cB_{\pmb\tau})$ the infimum of positive $\lambda$ such that $\lambda\cB_{\pmb\tau}$ contains $\vec v$, and set
\[L_{\vec v}(\pmb\tau)=\log\big(\lambda_{\vec v}(\cB_{\pmb\tau})\big).\]
If $\vec v=(v_1,\ldots,v_d)$, then
\[\lambda_{\vec v}(\cB_{\pmb\tau})=\max_{1\leq i\leq d}(|v_i|e^{-\tau_i}),\]
and
\[L_{\vec v}(\pmb\tau)=\max_{\substack{1\leq i\leq d \\ v_i\neq0}}\big(\log|v_i|-\tau_i\big),\]
i.e. $L_{\vec v}(\pmb\tau)$ is continuous and piecewise linear. Notice that for each $\pmb\tau$ and each $k=1,\ldots,d$ there is a $\vec v=\vec v(\pmb\tau,k)\in\La$ such that $\lambda_k(\cB_{\pmb\tau})=\lambda_{\vec v}(\cB_{\pmb\tau})$. Hence, denoting
\begin{equation} \label{eq:La_k}
  \La_k=\Big\{ \vec v\in\La\, \Big|\, \exists\,\pmb\tau:\,\lambda_k(\cB_{\pmb\tau})\leq\lambda_{\vec v}(\cB_{\pmb\tau}) \Big\},
\end{equation}
we get
\[L_k(\pmb\tau)=\min_{\vec v\in\La_k}L_{\vec v}(\pmb\tau).\]
Thus, $L_k(\pmb\tau)$ is indeed continuous and piecewise linear.
\end{proof}


\begin{proposition} \label{prop:properties_of_S_k}
  The functions $S_k(\pmb\tau)$ enjoy the following properties:

  \vskip1.5mm
  \textup{(i)} $-\log d!\leq S_d(\pmb\tau)\leq0$;

  \vskip1.5mm
  \textup{(ii)} $\dfrac{k+1}{k}S_k(\pmb\tau)\leq S_{k+1}(\pmb\tau)\leq\dfrac{d-k-1}{d-k}S_k(\pmb\tau)$;

  \vskip1.5mm
  \textup{(iii)} $(d-1)S_1(\pmb\tau)\leq S_{d-1}(\pmb\tau)\leq S_1(\pmb\tau)/(d-1)$;

  \vskip3mm
  \textup{(iv)} $S_{d-1}(\pmb\tau)=-L_d(\pmb\tau)+O(1)$.
\end{proposition}

\begin{proof}
  Statement \textup{(i)} follows from Minkowski's second theorem, which states that
  \[\frac1{d!}\leq\prod_{1\leq k\leq d}\lambda_k(\cB_{\pmb\tau})\leq1.\]

  Furthermore, statement \textup{(i)} of Proposition \ref{prop:properties_of_L_k} and statement \textup{(i)} of the current Proposition imply
  \[S_k(\pmb\tau)\leq kL_{k+1}(\pmb\tau)\]
  and
  \[S_k(\pmb\tau)+(d-k)L_{k+1}(\pmb\tau)\leq S_d(\pmb\tau)\leq0.\]
  Hence
  \[\frac1{k}S_k(\pmb\tau)\leq L_{k+1}(\pmb\tau)\leq\dfrac{-1}{d-k}S_k(\pmb\tau),\]
  and \textup{(ii)} follows.

  Applying statement \textup{(ii)} consequently, we get statement \textup{(iii)}.

  As for statement \textup{(iv)}, it is an immediate corollary of statement \textup{(i)}.
\end{proof}

We remind that $\La^\ast$ denotes the dual lattice.

\begin{proposition} \label{prop:L_and_S_for_dual_lattices}
  For each $\pmb\tau$ we have

  \vskip1.5mm
  \textup{(i)} 
  \hskip2.5mm
  $-\log d\leq L_k(\pmb\tau,\La)+L_{d+1-k}(-\pmb\tau,\La^\ast)\leq\log d!,\qquad\quad\ k=1,\ldots,d$,

  \vskip1.5mm
  \textup{(ii)} 
  $-k\log d\leq S_k(\pmb\tau,\La)-S_{d-k}(-\pmb\tau,\La^\ast)\leq(k+1)\log d!,\quad k=1,\ldots,d-1$.
\end{proposition}

\begin{proof}
  In his paper \cite{mahler_casopis_convex} Mahler proved that for a parallelepiped $\cB_{\pmb\tau}$ and its polar cross-polytope $\cB_{\pmb\tau}^\circ$ we have
  \begin{equation} \label{eq:mahler_for_cube_and_octahedron}
    1\leq\lambda_k(\cB_{\pmb\tau},\La)\lambda_{d+1-k}(\cB_{\pmb\tau}^\circ,\La^\ast)\leq d!.
  \end{equation}
  Since
  \[\cB_{\pmb\tau}^\circ=
    (D_{\pmb\tau}\cB)^\circ=
    D_{\pmb\tau}^{-1}\cB^\circ=
    D_{-\pmb\tau}\cB^\circ,\]
  we have
  \[d^{-1}\cB_{-\pmb\tau}\subset\cB_{\pmb\tau}^\circ\subset\cB_{-\pmb\tau}.\]
  Therefore, \eqref{eq:mahler_for_cube_and_octahedron} implies
  \[\frac1d\leq\lambda_k(\cB_{\pmb\tau},\La)\lambda_{d+1-k}(\cB_{-\pmb\tau},\La^\ast)\leq d!.\]
  Taking the $\log$ of all sides, we get \textup{(i)}.

  Furthermore, statement \textup{(i)} implies
  \begin{equation} \label{eq:L_for_dual_lattices_summed_up}
    -k\log d\leq S_k(\pmb\tau,\La)+\big(S_d(-\pmb\tau,\La^\ast)-S_{d-k}(-\pmb\tau,\La^\ast)\big)\leq k\log d!,
  \end{equation}
  whereas by statement \textup{(i)} of Proposition \ref{prop:properties_of_S_k}
  \begin{equation} \label{eq:S_d_for_dual_lattice}
    0\leq-S_d(-\pmb\tau,\La^\ast)\leq\log d!.
  \end{equation}
  Summing up \eqref{eq:L_for_dual_lattices_summed_up} and \eqref{eq:S_d_for_dual_lattice}, we get \textup{(ii)}.
\end{proof}

\subsection{Local essence of the transference phenomenon} 

There is a large variety of transference theorems for all kinds of Diophantine exponents. All of them depict the connection between problems that are in some sense dual. We claim that Propositions \ref{prop:properties_of_S_k} and \ref{prop:L_and_S_for_dual_lattices} provide a relation of local nature which implies most of the existing transference theorems.

\begin{theorem} \label{t:essence_of_transference}
  We have
  \begin{equation} \label{eq:essence_of_transference_leq}
    S_1(\pmb\tau,\La)\leq\frac{S_1(-\pmb\tau,\La^\ast)}{d-1}+O(1).     
  \end{equation}
\end{theorem}

Theorem \ref{t:essence_of_transference} is a corollary to the following statement, which in addition splits \eqref{eq:essence_of_transference_leq} into a chain of inequalities between the values of consecutive $S_k$.

\begin{theorem} \label{t:essence_of_transference_split_up}
  We have

  \textup{(i)}
  $S_k(\pmb\tau,\La)=S_{d-k}(-\pmb\tau,\La^\ast)+O(1),\ \ k=1,\ldots,d-1,$ \vphantom{$\frac{\big|}{|}$}

  \textup{(ii)}
  $\displaystyle
    S_1(\pmb\tau,\La)\leq\ldots\leq
    \frac{S_k(\pmb\tau,\La)}{k}\leq\ldots\leq
    \frac{S_{d-1}(\pmb\tau,\La)}{d-1}.$

  \textup{(iii)}
  $\displaystyle
    \frac{S_1(\pmb\tau,\La)}{d-1}\geq\ldots\geq
    \frac{S_k(\pmb\tau,\La)}{d-k}\geq\ldots\geq
    S_{d-1}(\pmb\tau,\La).$
%
\end{theorem}

\begin{proof}
  Statement \textup{(i)} follows from statement \textup{(ii)} of Proposition \ref{prop:L_and_S_for_dual_lattices}. Statements \textup{(ii)} and \textup{(iii)} follow from statement \textup{(ii)} of Proposition \ref{prop:properties_of_S_k}.
\end{proof}

It is clear that \eqref{eq:essence_of_transference_leq} follows immediately from statements \textup{(i)} and \textup{(ii)} of Theorem \ref{t:essence_of_transference_split_up}. Notice that the inequality
\begin{equation} \label{eq:essence_of_transference_geq}
  S_1(\pmb\tau,\La)\geq(d-1)S_1(-\pmb\tau,\La^\ast)+O(1)
\end{equation}
provided in the same way by statements \textup{(i)} and \textup{(iii)} of Theorem \ref{t:essence_of_transference_split_up} is actually the same as \eqref{eq:essence_of_transference_leq}, since for each $\La\in\cL$ we have $(\La^\ast)^\ast=\La$.

We also notice that, when formulated separately, it might be more natural to write $L_1$ instead of $S_1$ in \eqref{eq:essence_of_transference_leq} and \eqref{eq:essence_of_transference_geq}. However, in connection with Theorem \ref{t:essence_of_transference_split_up}, we prefer to write $S_1$.

\subsection{Schmidt--Summerer exponents} 

Every norm in $\R^d$ induces a norm in $\cT$. Particularly, the sup-norm $|\cdot|$. As for the functionals induced by $|\cdot|_+$ and $|\cdot|_-$, they are not norms, the corresponding ``unit balls'' are simplices and are not $\vec 0$-symmetric. However, $|\cdot|_+$ cannot be neglected, as it is the image of the sup-norm under the logarithmic mapping: if $\vec x=(x_1,\ldots,x_d)$, $x_i>0$, $i=1,\ldots,d$, and $\vec x_{\log}=(\log x_1,\ldots,\log x_d)$, then
\[\log|\vec x|=|\vec x_{\log}|_+\,.\]
Most of our argument works for an arbitrary functional one can choose to measure $\pmb\tau$, it only needs to generate an exhaustion of $\cT$.

Let $f$ be an arbitrary non-negative function on $\cT$ such that the sets
\[\Big\{ \pmb\tau\in\cT\, \Big|\,f(\pmb\tau)\leq\lambda \Big\}\]
form a monotone exhaustion of $\cT$.
Set
\[f^\ast(\pmb\tau)=f(-\pmb\tau).\]
For each $k=1,\ldots,d$ consider the functions
\[
  \psi_k(\pmb\tau,\La,f)=
  \frac{L_k(\pmb\tau,\La)}{f(\pmb\tau)}\,,
  \qquad
  \Psi_k(\pmb\tau,\La,f) =
  \sum_{1\leq j\leq k}\psi_j(\pmb\tau,\La,f).
\]

\begin{definition}
  Given $\La$ and $f$, we define the \emph{Schmidt--Summerer lower and upper exponents of the first type} as
  \[\bpsi_k(\La,f)=\liminf_{|\pmb\tau|\to\infty}\psi_k(\pmb\tau,\La,f),
    \qquad
    \apsi_k(\La,f)=\limsup_{|\pmb\tau|\to\infty}\psi_k(\pmb\tau,\La,f),\]
  and \emph{of the second type} as
  \[\bPsi_k(\La,f)=\liminf_{|\pmb\tau|\to\infty}\Psi_k(\pmb\tau,\La,f),
    \qquad
    \aPsi_k(\La,f)=\limsup_{|\pmb\tau|\to\infty}\Psi_k(\pmb\tau,\La,f).\]
\end{definition}

As before, in the absence of ambiguity, we will omit sometimes the indication of dependence on $\La$ and $f$.

Clearly, $|\pmb\tau|\to\infty$ if and only if $f(\pmb\tau)\to\infty$. Hence, Propositions \ref{prop:properties_of_L_k}, \ref{prop:properties_of_S_k}, \ref{prop:L_and_S_for_dual_lattices}, after division by $f(\pmb\tau)$ and taking the $\liminf$'s and $\limsup$'s, provide the following variety of relations.

\begin{proposition} \label{prop:inequalities_for_schmexponents}
  Given $\La$ and $f$, we have
  \vskip-8mm
  \[
  \begin{aligned}
    & \textup{(i)}    & \bpsi_k\leq\bpsi_{k+1}\,, & \qquad \vphantom{\bigg|}
                        \apsi_k\leq\apsi_{k+1}\,; \\
    & \textup{(ii)}   & \bPsi_1=\bpsi_1\leq0\,, & \qquad
                        \aPsi_1=\apsi_1\leq0\,; \\
    & \textup{(iii)}  & \bPsi_{d-1}=-\apsi_d\,, & \qquad \vphantom{\bigg|}
                        \aPsi_{d-1}=-\bpsi_d\,; \\
    & \textup{(iv)}   & \bPsi_d=0\,, &  \qquad
                        \aPsi_d=0\,;  \\
    & \textup{(v)}    & \dfrac{k+1}{k}\bPsi_k\leq\bPsi_{k+1}\leq\dfrac{d-k-1}{d-k}\bPsi_k\leq0\,, & \qquad \vphantom{\bigg|}
                        \dfrac{k+1}{k}\aPsi_k\leq\aPsi_{k+1}\leq\dfrac{d-k-1}{d-k}\aPsi_k\leq0\,; \\
    & \textup{(vi)}   & (d-1)\bPsi_1\leq\bPsi_{d-1}\leq\bPsi_1/(d-1)\,, & \qquad
                        (d-1)\aPsi_1\leq\aPsi_{d-1}\leq\aPsi_1/(d-1)\,; \\
    & \textup{(vii)}  & \bpsi_k(\La,f)=-\apsi_{d+1-k}(\La^\ast,f^\ast)\,, & \qquad \vphantom{\bigg|}
                        \apsi_k(\La,f)=-\bpsi_{d+1-k}(\La^\ast,f^\ast)\,; \\
    & \textup{(viii)} & \bPsi_k(\La,f)=\bPsi_{d-k}(\La^\ast,f^\ast)\,, & \qquad
                        \aPsi_k(\La,f)=\aPsi_{d-k}(\La^\ast,f^\ast)\,.
  \end{aligned}
  \]
\end{proposition}

\section{Lattice exponents in terms of Schmidt--Summerer exponents}\label{sec:lattice_exponents_schimmerered}

Schmidt--Summerer exponents of lattices are, in a sense, global characteristics, whereas we could consider a one-parametric path $\big\{\pmb\tau(s)\,\big|\,s\in\R_+\big\}$ and the corresponding $\liminf$'s and $\limsup$'s as $s\to\infty$. This is performed in \cite{german_AA_2012} for the path defined by
\[\tau_1(s)=\ldots=\tau_m(s)=s,\quad\tau_{m+1}(s)=\ldots=\tau_d(s)=-ms/n\]
corresponding to the problem of simultaneous approximation of zero with the values of $n$ linear forms in $m$ variables, $n+m=d$. In that case Schmidt--Summerer exponents correspond to \emph{intermediate Diophantine exponents} (see \cite{german_AA_2012}). In the current setting we have a similar situation: the exponents $\bPsi_1(\La,f)=\bpsi_1(\La,f)$ for $f(\pmb\tau)=|\pmb\tau|_+$ and $\omega(\La)$ are but two different points of view at the same phenomenon.


\begin{proposition} \label{prop:omega_vs_Psi}
  Let $f(\pmb\tau)=|\pmb\tau|_+$. Then
  \[\omega(\La)^{-1}+\bPsi_1(\La,f)^{-1}+1=0.\]
\end{proposition}

\begin{proof}
  As in the proof of Proposition \ref{prop:properties_of_L_k}, let us set for each nonzero $\vec v\in\La$ and each $\pmb\tau\in\cT$
  \[
  \begin{array}{l}
    \lambda_{\vec v}(\cB_{\pmb\tau})=\inf\Big\{ \lambda>0\, \Big|\,\vec v\in\lambda\cB_{\pmb\tau} \Big\}, \\
    L_{\vec v}(\pmb\tau)=\log\big(\lambda_{\vec v}(\cB_{\pmb\tau})\big). \vphantom{^{\big|}}
  \end{array}
  \]
  For each $\vec v=(v_1,\ldots,v_d)\in\La$ let us set
  \[\pmb\tau(\vec v)=\Big(\log\big(|v_1|\big/\Pi(\vec v)\big),\ldots,\log\big(|v_d|\big/\Pi(\vec v)\big)\Big).\]
  Then
  \[
  \begin{array}{l}
    \lambda_{\vec v}(\cB_{\pmb\tau(\vec v)})=\Pi(\vec v), \\ \vphantom{\bigg|}
    L_{\vec v}(\pmb\tau(\vec v))=\log(\Pi(\vec v)), \\
    |\pmb\tau(\vec v)|_+=\log|\vec v|-\log(\Pi(\vec v)).
  \end{array}
  \]
  Hence
  \begin{multline*}
    \bPsi_1(\La,f)=
    \bpsi_1(\La,f)=
    \liminf_{|\pmb\tau|\to\infty}\frac{L_1(\pmb\tau)}{|\pmb\tau|_+}=
    \liminf_{|\pmb\tau|\to\infty}\frac{\min_{\vec v\in\La}L_{\vec v}(\pmb\tau)}{|\pmb\tau|_+}=\vphantom{\Bigg|} \\ =
    \liminf_{\substack{\vec v\in\La \\ |\vec v|\to\infty}}\frac{L_{\vec v}(\pmb\tau(\vec v))}{|\pmb\tau(\vec v)|_+}=
    \liminf_{\substack{\vec v\in\La \\ |\vec v|\to\infty}}\frac{\log(\Pi(\vec v))}{\log|\vec v|-\log(\Pi(\vec v))}=\quad \\ =
    -\Bigg(1+\Bigg(\limsup_{\substack{\vec v\in\La \\ |\vec v|\to\infty}}\frac{\log\big(\Pi(\vec v)^{-1}\big)}{\log|\vec v|}\Bigg)^{-1}\Bigg)^{-1}=
    -\Big(1+\omega(\La)^{-1}\Big)^{-1}.
  \end{multline*}
\end{proof}

\section{Splitting the transference theorem}\label{sec:splitting}

Proposition \ref{prop:omega_vs_Psi} allows reformulating Theorem \ref{t:lattice_transference} in terms of Schmidt--Summerer exponents. Remark \ref{r:lattice_transference} makes this reformulation very easy to perform.

\begin{theorem}[Reformulation of Theorem \ref{t:lattice_transference}] \label{t:lattice_transference_schmimmerered}
  Let $f(\pmb\tau)=|\pmb\tau|_+$. Then
  \begin{equation} \label{eq:lattice_transference_schmimmerered}
    \bPsi_1(\La,f)\leq\frac{\bPsi_1(\La^\ast,f)}{(d-1)^2}.
  \end{equation}
  (Notice that by statement \textup{(v)} of Proposition \ref{prop:inequalities_for_schmexponents} $\bPsi_1$ is never positive.)
\end{theorem}

It appears that due to Proposition \ref{prop:inequalities_for_schmexponents} Theorem \ref{t:lattice_transference_schmimmerered} can be split in the very same way Theorem \ref{t:essence_of_transference} is split by Theorem \ref{t:essence_of_transference_split_up}.

\begin{theorem}[Splitting Theorem \ref{t:lattice_transference_schmimmerered}] \label{t:lattice_transference_schmimmerered_split_up}
  Let $f(\pmb\tau)=|\pmb\tau|_+$. Then
  \[
    \bPsi_1(\La,f)\leq\ldots\leq
    \frac{\bPsi_k(\La,f)}{k}\leq\ldots\leq
    \frac{\bPsi_{d-1}(\La,f)}{d-1}\leq
    \dfrac{\bPsi_1(\La^\ast,f)}{(d-1)^2}.
  \]
\end{theorem}

\begin{proof}
  It suffices to apply statements \textup{(v)}, \textup{(viii)} of Proposition \ref{prop:inequalities_for_schmexponents}, and notice that by \eqref{eq:tau_minus_vs_tau_plus} we have
  \[
    \frac{f(\pmb\tau)}{d-1}\leq f^\ast(\pmb\tau)\leq(d-1)f(\pmb\tau),
  \]
  whence it follows that
  \[
    (d-1)\bPsi_1(\La^\ast,f)\leq
    \bPsi_1(\La^\ast,f^\ast)\leq
    \frac{\bPsi_1(\La^\ast,f)}{d-1}.
  \]
\end{proof}

It can be easily seen that the argument used in the proof of Theorem \ref{t:lattice_transference_schmimmerered_split_up} actually provides a stronger version of Theorem \ref{t:lattice_transference_schmimmerered}. If we just omit the last step, and do not substitute $f^\ast$ with $f$, we get the following statement.

\begin{theorem} \label{t:lattice_transference_for_f_congugate}
  For every $f$ we have
  \begin{equation} \label{eq:lattice_transference_for_f_congugate}
    \bPsi_1(\La,f)\leq\frac{\bPsi_1(\La^\ast,f^\ast)}{d-1}.
  \end{equation}
\end{theorem}

\begin{remark}
 The appearance of Theorem \ref{t:lattice_transference_for_f_congugate} becomes nicest when $f$ is symmetric, for instance, when $f(\pmb\tau)=|\pmb\tau|$. However, the corresponding reformulation in terms of Diophantine exponents of lattices requires another definition of those exponents, which might seem less natural.
\end{remark}

\section{Three-dimensional case and minimal systems of lattice points}\label{sec:3dim}

Let us analyse the relations gathered up in Proposition \ref{prop:inequalities_for_schmexponents} in the simplest nontrivial case $d=3$.

There are two major cases for $d=3$, depending on whether $\La$ contains nonzero points of each coordinate axis, or not. In other words, whether $\La$ is the image of a sublattice of $\Z^3$ under the action of a diagonal operator, or not. If it is, the exponents of $\La$ are equal to those of $\Z^3$, whereas the exponents of $\Z^3$ are easily calculated. But if it is not, there are infinitely many \emph{Minkowski bases} of $\La$. 
%
%
%
%
%
%
A basis of $\La$ consisting of vectors $\vec v_j=(v_{j1},v_{j2},v_{j3})$, $j=1,2,3$, is a \emph{Minkowski basis}, if there is no nonzero $\vec v=(v_1,v_2,v_3)\in\La$ such that for each $i=1,2,3$ we have
\[|v_i|<\max(|v_{1i}|,|v_{2i}|,|v_{3i}|).\]
In this case we obviously have $\apsi_1=\aPsi_1=0$, so, by Proposition \ref{prop:inequalities_for_schmexponents}
\[
\begin{aligned}
  & \apsi_1=\bpsi_3=\aPsi_1=\aPsi_2=\aPsi_3=\bPsi_3=0, \\
  & \bpsi_1=\bPsi_1,\hskip3.9mm
    \apsi_3=-\bPsi_2, \\
  & \bpsi_1(\La,f)=-\apsi_3(\La^\ast,f^\ast),\hskip3.9mm
    \apsi_3(\La,f)=-\bpsi_1(\La^\ast,f^\ast), \\
  & \bpsi_2(\La,f)=-\apsi_2(\La^\ast,f^\ast),\hskip3.9mm
    \apsi_2(\La,f)=-\bpsi_2(\La^\ast,f^\ast),
\end{aligned}
\]
Hence $\bpsi_1,\bpsi_2,\apsi_2,\apsi_3$ determine the remaining exponents. 
They also satisfy
\begin{equation} \label{eq:bpsi_1_2_apsi_2_3}
  c_1(f)\leq\bpsi_1\leq\bpsi_2\leq0\leq\apsi_2\leq\apsi_3\leq c_2(f)
\end{equation}
for some constants $c_1(f),c_2(f)$ (see Proposition \ref{prop:properties_of_L_k}). For instance,
\[c_1(|\cdot|_+)=-1,\quad c_2(|\cdot|_+)=d-1.\]

It is natural to ask if the set of all possible values of the quadruple
\[(\bpsi_1,\bpsi_2,\apsi_2,\apsi_3)\]
is determined by \eqref{eq:bpsi_1_2_apsi_2_3}, or if there are other restrictions on those exponents.

It is also interesting whether in the case $d\geq4$ we have $\apsi_1(\La)=0$ for each lattice that is ``irrational'' enough. For instance, if there are no nonzero lattice points in the coordinate planes.

If $\apsi_1(\La)>0$, then all \emph{minimal systems} of lattice points, except maybe a finite number of them, have rank less than $d$. A system of $k$ lattice points $\vec v_j=(v_{j1},\ldots,v_{jd})$, $j=1,\ldots,k$, $k\leq d$, is called \emph{minimal}, if, same as for Minkowski bases, there is no nonzero $\vec v=(v_1,\ldots,v_d)\in\La$ such that for each $i=1,\ldots,d$ we have
\[|v_i|<\max_{1\leq j\leq n}|v_{ji}|.\]
The \emph{rank} of this system is the dimension of the minimal subspace containing it. Clearly, if the rank of $\vec v_1,\ldots,\vec v_d$ is $d$, they form a basis of a full rank sublattice of $\La$. If such a basis is a minimal system, the volume of the corresponding parallelepiped is bounded away from zero. And if there are infinitely many such systems, we obviously have $\apsi_1(\La)=0$.

On the other hand, if $\apsi_1(\La)=0$, then all the exponents $\aPsi_1,\ldots,\aPsi_d$ degenerate into zero, as in this case we have
\[0=\apsi(\La)=\aPsi_1(\La)\leq\ldots\leq\aPsi_d(\La)=0.\]
From this point of view the question whether there exists a lattice such that all its minimal systems have rank less than $d$ is of obvious importance.

%


\paragraph{Acknowledgements.}

The author is a Young Russian Mathematics award winner and would like to thank its sponsors and jury.

%
%
%
%
%

\end{document}